\documentclass[12pt, reqno]{amsart}
\makeatletter
\@namedef{subjclassname@1991}{$\mathrm{1991}$ Mathematics Subject Classification}
\@namedef{subjclassname@2000}{$\mathrm{2000}$ Mathematics Subject Classification}
\@namedef{subjclassname@2010}{$\mathrm{2010}$ Mathematics Subject Classification}
\@namedef{subjclassname@2020}{$\mathrm{2020}$ Mathematics Subject Classification}
\makeatother
\usepackage{amsmath,amsthm, amscd, amsfonts, amssymb, graphicx, color}
\usepackage[bookmarksnumbered, colorlinks, plainpages,linkcolor=blue,urlcolor=blue,citecolor=blue]{hyperref}
\newtheorem{thm}{Theorem}[section]
\newtheorem{cor}[thm]{Corollary}

\newtheorem{defn}[thm]{Definition}

\newtheorem{exam}[thm]{Example}
\numberwithin{equation}{section}

\begin{document}

\title{the $m$-generalized right group inverses in Banach algebras}

\author{Huanyin Chen}
\author{Marjan Sheibani}
\address{School of Big Data, Fuzhou University of International Studies and Trade, Fuzhou 350202, China}
\email{<huanyinchenfz@163.com>}
\address{Farzanegan Campus, Semnan University, Semnan, Iran}
\email{<m.sheibani@semnan.ac.ir>}


\begin{abstract} In this paper, we introduce the concept of the $m$-generalized right group inverse. This serves as a natural extension of both the $m$-weak group inverse and the generalized group inverse. This is the first time to study the one-side version of weak group inverse. We characterize this new generalized inverse using the $m$-generalized right group decomposition and a polar-like property. Additionally, we present the representation of the $m$-generalized right group inverse using the generalized right core inverse, leading to new insights and properties for both the $m$-weak group inverse and the generalized group inverse.\end{abstract}

\maketitle

\section{Introduction}

A Banach *-algebra is an algebra $\mathcal{A}$ that is also a Banach space with a norm $||\cdot ||$, and has an involution
$*$ satisfying: $(x+y)^*=x^*+y^*, (\lambda x)^*=\overline{\lambda} x^*, (xy)^*=y^*x^*$ and $(x^*)^*=x$ for all $x,y\in \mathcal{A}$.
The involution $*$ is proper if $x^*x=0\Longrightarrow x=0$ for any $x\in \mathcal{A}$. Every $C^*$-algebra is a Banach *-algebra with a proper involution.
In particular, the algebra ${\Bbb C}^{n\times n}$ consisting of all $n\times n$ complex matrices is a Banach *-algebra, where the proper involution is given by the conjugate transpose. Throughout the paper, $\mathcal{A}$ denotes a Banach *-algebra with a proper involution.

An element $a\in \mathcal{A}$ has a group inverse if there exists an
$x\in \mathcal{A}$ such that $$xa^2=a, ax^2=x, ax=xa.$$
Such an $x$ , if it exists, is unique and is denoted by
$a^{\#}$, termed the group inverse of $a$. A square complex matrix $A$ has a group inverse if $rank(A)=rank(A^2)$.
The group inverse is widely used in matrix and operator theory (see~\cite{C3,M}).

The concept of the weak group inverse is a significant generalization of the group inverse. An element $a\in \mathcal{A}$ has a weak group inverse there exist $x\in R$ and $n\in \Bbb{N}$ such that $$ax^2=x, (a^*a^2x)^*=a^*a^2x, xa^{n+1}=a^n.$$ If such $x$ exists, it is unique, and denote it by $a^{\tiny\textcircled{W}}$. Evidently, a square complex matrix $A$ possesses a group inverse if and only if the system of equations $$AX^2=X, AX=A^{\tiny\textcircled{\dag}}A$$ is solvable. Here, $A^{\tiny\textcircled{\dag}}$ stands for the core-EP inverse of $A$ (see~\cite{GC,M0}).
The weak group inverse extends the applicability of group inverses to broader classes of matrices or operators. It has found applications in solving singular systems, perturbation analysis, and constrained optimization (see~\cite{CY,D,L,M1,W,YW,Z2,Z3}).

The $m$-weak group inverse is a generalization of the weak group inverse and is useful in various applications, e.g. ~\cite{M2,M3}.
An element $a\in \mathcal{A}$ has $m$-weak group inverse if there exist $x\in \mathcal{A}$ and $k\in {\Bbb N}$ such that $$ax^2=x, xa^{k+1}=a^k~\mbox{and} ~(a^k)^*a^{m+1}x=(a^k)^*a^m.$$ The preceding $x$ is called the $m$-weak group inverse of $a$, and denoted by $a^{\tiny\textcircled{W}_m}$. Evidently, a square complex matrix $A$ admits an $m$-weak group inverse $X$, provided it satisfies the following system of equations: $$AX^2=X, AX=(A^{\tiny\textcircled{\dag}})^mA^m $$ (see~\cite{G,J,L,M5,Z4}).

Following Chen, an element $a\in \mathcal{A}$ has $m$-generalized group inverse if there exists $x\in \mathcal{A}$ such that $$ax^2=x, ((a^m)^*a^{m+1}x)^*=(a^m)^*a^{m+1}x, \lim\limits_{n\to \infty}||a^n-xa^{n+1}||^{\frac{1}{n}}=0.$$ Such $x$ is unique if exists, denoted by $a^{\tiny\textcircled{g}}.$ Many properties of $m$-generalized group inverse were investigated in ~\cite{C1}. For a complex matrix, the
$m$-generalized group inverse and $m$-weak group inverse coincide. However, the
$m$-generalized group inverse extends the applicability of $m$-weak group inverses to broader classes of matrices or operators,
including those over infinite-dimensional Hilbert spaces (see~\cite{M4}).

Following Yan (see~\cite{Y}), an element $a\in \mathcal{A}$ has right group inverse provided that there exists $x\in \mathcal{A}$ such that $$ax^2=x, a^2x=axa=a.$$ We use $\mathcal{A}_r^{\#}$ to denote the set of all right group invertible elements in $\mathcal{A}$.

An $a\in \mathcal{A}$ has generalized right Drazin inverse if there exists $x\in \mathcal{A}$ such that $$ax^2=x, a^2x=axa, a-axa\in \mathcal{A}^{qnil}.$$ Here, $$\mathcal{A}^{qnil}=\{x\in \mathcal{A}~\mid~ \lim\limits_{n\to \infty}\parallel x^n\parallel^{\frac{1}{n}}=0\}.$$ As is well known, $x\in \mathcal{A}^{qnil}$ if and only if $1+\lambda x\in \mathcal{A}$ is invertible. We use $\mathcal{A}_r^{d}$ to denote the set of all generalized right Drazin invertible elements in $\mathcal{A}$. Many characterizations of generalized right (left) Drazin inverse are established in ~\cite{R,Y}.

Many matrices do not admit a weak group inverse. To address this limitation, the generalized right group inverse was introduced as a relaxation that retains some desirable properties of the weakly group inverse while being applicable to a broader class of matrices. We refer the reader to~\cite{CM2} for more properties of such kind of generalized inverses.

The motivation of this paper is to introduce and study a new type of generalized inverse, serving as a natural extension of the $m$-weak group inverse, with the aim of applying it to infinite-dimensional bounded linear operators on Hilbert spaces. This is the first time to study the one-side version of the weak group inverse.

\begin{defn} An element $a\in \mathcal{A}$ has an $m$-generalized right group inverse if $a\in \mathcal{A}_r^d$ and there exists $x\in \mathcal{A}$ such that $$x=ax^2, (aa_r^d)^*a^{m+1}x=(aa_r^d)^*a^m, \lim\limits_{n\to \infty}||a^n-axa^n||^{\frac{1}{n}}=0.$$ The preceding $x$ is called an $m$-generalized right group inverse of $a$, and denoted by $a_r^{\tiny\textcircled{g}_m}$.\end{defn}

Section 2 establishes the fundamental properties of this new generalized inverse, revealing several new characteristics of the $m$-weak group inverse.
We prove that $a\in \mathcal{A}_r^{\tiny\textcircled{g}_m}$ if and only if $a^m\in \mathcal{A}_r^{\tiny\textcircled{g}}$ if and only if
$a\in \mathcal{A}_r^d$ and $(aa_r^d)^*aa_r^dx=(aa_r^d)^*a^m$ for some $x\in \mathcal{A}$.

In Section 3, we prove that $a\in \mathcal{A}_r^{\tiny\textcircled{g}_m}$ if and only if there exists $x\in \mathcal{A}$ such that
$$x=ax^2, [(a^m)^*a^{m+1}x]^*=(a^m)^*a^{m+1}x, \lim\limits_{n\to \infty}||a^n-axa^{n}||^{\frac{1}{n}}=0,$$
if and only if $a$ has an $m$-generalized group decomposition, i.e., there exist $x,y\in \mathcal{A}$ such that $$a=x+y, x^*a^{m-1}y=yx=0, x\in \mathcal{A}_r^{\#}, y\in \mathcal{A}^{qnil}.$$ We present the polar-like property for $m$-generalized right group inverse. We prove that
$a\in \mathcal{A}_r^{\tiny\textcircled{g}_m}$ if and only if there exists an idempotent $p\in \mathcal{A}$ such that
$$\begin{array}{c}
(1-p)a(1-p)\in [(1-p)\mathcal{A}(1-p)]_r^{-1}, \big((a^m)^*a^{m}p\big)^*=(a^m)^*a^{m}p,\\
ap\in \mathcal{A}^{qnil}, (1-p)\mathcal{A}=a(1-p)\mathcal{A},
\end{array}$$ if and only if there exists $b\in \mathcal{A}$ such that $$bab=b, a^2b^2=ab, [(a^m)^*a^{m+1}b]^*=(a^m)^*a^{m+1}b, ab\mathcal{A}=a^2b\mathcal{A}, a-a^2b\in \mathcal{A}^{qnil}.$$ This yields new characterizations of the $m$-weak group inverse.

 An element $a\in \mathcal{A}$ has generalized right core inverse if there exists $x\in \mathcal{A}$ such that $$ax^2=x, (ax)^*=ax, \lim_{n\to
\infty}||a^n-axa^{n+1}||^{\frac{1}{n}}=0.$$ If such a $x$ exists, we denote it by $a_r^{\tiny\textcircled{d}}$ (see~\cite{CM2}).

Finally, in Section 4, we investigate the representations of $m$-generalized right group inverses using the generalized right core inverse. We prove that every element with a generalized right core inverse possesses an $m$-generalized right group inverse and
$$a_r^{\tiny\textcircled{g}_m}=(a_r^d)^{m+1}aa_r^{\tiny\textcircled{d}}a^{m}=(a_r^daa_r^{\tiny\textcircled{d}})^{m+1}a^m.$$ Applications in solving matrix equations are thus obtained.

Throughout the paper, $\mathcal{A}_r^{\#}, \mathcal{A}_r^{d}, \mathcal{A}_r^{\tiny\textcircled{d}}, \mathcal{A}_r^{\tiny\textcircled{g}}$, $\mathcal{A}^{\tiny\textcircled{W}_m}$ and $\mathcal{A}_r^{\tiny\textcircled{g}_m}$ denote the sets of all right group invertible, generalized right Drazin invertible, generalized right core invertible, generalized right group invertible, $m$-weak group invertible and  and $m$-generalized right group invertible elements in $\mathcal{A}$, respectively. ${\Bbb N}$ is the set of all natural numbers.

\section{$m$-generalized right group inverse}

In this section, we explore the fundamental properties of the $m$-generalized right group inverse of elements in a Banach *-algebra. We start with

\begin{thm} Let $a\in \mathcal{A}$. Then $a\in \mathcal{A}_r^{\tiny\textcircled{g}_m}$ if and only if $a^m\in \mathcal{A}_r^{\tiny\textcircled{g}}$. In this case, $a_r^{\tiny\textcircled{g}_m}=a^{m-1}(a^m)_r^{\tiny\textcircled{g}}$ and
$(a^m)_r^{\tiny\textcircled{g}}=(a_r^{\tiny\textcircled{g}_m})^m$.\end{thm}
\begin{proof} $\Longrightarrow$ Since $a\in \mathcal{A}_r^{\tiny\textcircled{g}_m}$, then $a\in \mathcal{A}_r^d$ and there exists $x\in \mathcal{A}$ such that $$x=ax^2, (aa_r^d)^*a^{m+1}x=(aa_r^d)^*a^m, \lim\limits_{n\to \infty}||a^n-axa^n||^{\frac{1}{n}}=0.$$
Let $y=x^m$. Then we verify that
$$\begin{array}{rll}
a^my^2&=&a^mx^{2m} =(ax)x^m=(ax^2)x^{m-1}=x^m=y,\\
(a^m(a^m)_r^d)^*(a^m)^2y&=&(a^m(a^m)_r^d)^*a^m(a^mx^m)=(a^m(a^m)_r^d)^*a^{m+1}x\\
&=&(aa_r^da^m(a^m)_r^d)^*a^{m+1}x=(a^m(a^m)_r^d)^*[(aa_r^d)^*a^{m+1}x]\\
&=&(a^m(a^m)_r^d)^*[(aa_r^d)^*a^{m}]=(aa_r^da^m(a^m)_r^d)^*a^{m}=(a^m(a^m)_r^d)^*a^m,\\
\end{array}$$
Moreover, we have
$$\begin{array}{rl}
&(a^m)^n-(a^m)y(a^m)^{n}=(a^m)^n-(a^m)x^m(a^m)^{n}\\
=&(a^m)^n-ax(a^m)^{n}\\
=&a^{mn}-axa^{mn}.
\end{array}$$ Hence, we have
$$\begin{array}{rl}
&||(a^m)^n-(a^m)y(a^m)^{n}||^{\frac{1}{n}}\\
\leq &\big(||a^{mn}-axa^{mn}||^{\frac{1}{mn}}\big)^m.
\end{array}$$
Therefore $$\lim\limits_{n\to \infty}||(a^m)^n-(a^m)y(a^m)^{n}||^{\frac{1}{n}}=0.$$ By virtue of~\cite[Theorem 2.6]{CM}, $a^m\in \mathcal{A}_r^{\tiny\textcircled{g}}$. Then
$(a^m)_r^{\tiny\textcircled{g}}=y=x^m$, as required.

$\Longleftarrow $ Let $y=(a^m)_r^{\tiny\textcircled{g}}$. Then $a^m\in \mathcal{A}_r^d$ and we have
$$\begin{array}{rll}
a^my^2&=&y,\\
((a^m)_r^d)^*(a^m)^2y&=&((a^m)_r^d)^*a^m,\\
\lim\limits_{n\to \infty}||(a^m)^n-(a^m)y(a^m)^{n}||^{\frac{1}{n}}&=&0.
\end{array}$$ Take $x=a^{m-1}y$. Then we verify that
$$\begin{array}{rll}
ax^2&=&a[a^{m-1}ya^{m-1}]y=a^mya^{m-1}a^m[(a^m)_r^{\tiny\textcircled{g}}]^2\\
&=&a^ma^{m-1}y^2=a^{m-1}y=x,\\
(aa_r^d)^*a^{m+1}x&=&(aa_r^d)^*a^{2m}y=(a^{m+1})^*((a^m)_r^d)^*a^{2m}y\\
&=&(a^{m+1})^*((a^m)_r^d)^*a^m=(aa_r^d)^*a^m.
\end{array}$$
Obviously, we have
$$a^{mn}-axa^{mn}=a^{mn}-a[a^{m-1}y]a^{mn}=(a^{m})^n-a^{m}y(a^{m})^n.$$
Hence,
$$\begin{array}{rl}
&||a^{mn}-axa^{mn}||^{\frac{1}{mn}}\\
\leq &\big(||(a^{m})^n-a^{m}y(a^{m})^n||^{\frac{1}{n}}\big)^{\frac{1}{m}}.
\end{array}$$ Thus, we have
$$\lim\limits_{n\to \infty}||a^{mn}-axa^{mn}||^{\frac{1}{mn}}=0.$$ We directly check that $a\in \mathcal{A}_r^d$ and $a_r^d=a^{m-1}(a^m)_rZ^d$.
Accordingly, $a_r^{\tiny\textcircled{g}_m}=x=a^{m-1}(a^m)_r^{\tiny\textcircled{g}}$, as required.\end{proof}

\begin{cor} Let $a\in \mathcal{A}_r^{\tiny\textcircled{g}_m}$. Then $(a^m)_r^{\tiny\textcircled{g}}=(a_r^{\tiny\textcircled{g}_m})^m.$\end{cor}
\begin{proof} This is obvious by the proof of Theorem 2.1.\end{proof}

\begin{cor} Let $a,b\in \mathcal{A}_r^{\tiny\textcircled{g}_m}$. If $ab=ba=a^*b=0$, then $a+b\in \mathcal{A}_r^{\tiny\textcircled{g}_m}$. In this case,
$$(a+b)_r^{\tiny\textcircled{g}_m}=a_r^{\tiny\textcircled{g}_m}+b_r^{\tiny\textcircled{g}_m}.$$\end{cor}
\begin{proof} By virtue of Theorem 2.1, $a^m,b^m\in \mathcal{A}_r^{\tiny\textcircled{g}}$ and
$$a_r^{\tiny\textcircled{g}_m}=x=a^{m-1}(a^m)_r^{\tiny\textcircled{g}}, b_r^{\tiny\textcircled{g}_m}=b^{m-1}(b^m)_r^{\tiny\textcircled{g}}.$$
By hypothesis, we see that $(a^m)(b^m)=(b^m)(a^m)=(a^m)^*(b^m)=0$. According to~\cite[Corollary 2.8]{CM}, $(a+b)^m=a^m+b^m\in \mathcal{A}_r^{\tiny\textcircled{g}}$.
By virtue of Theorem 2.1 again, $a+b\in \mathcal{A}_r^{\tiny\textcircled{g}_m}$. Moreover, we have
$$\begin{array}{rll}
(a+b)_r^{\tiny\textcircled{g}_m}&=&(a+b)^{m-1}\big((a+b)^m\big)_r^{\tiny\textcircled{g}}\\
&=&(a^{m-1}+b^{m-1})\big((a^m)_r^{\tiny\textcircled{g}}+(b^m)_r^{\tiny\textcircled{g}}\big)\\
&=&a^{m-1}(a^m)_r^{\tiny\textcircled{g}}+b^{m-1}(b^m)_r^{\tiny\textcircled{g}}\\
&=&a_r^{\tiny\textcircled{g}_m}+b_r^{\tiny\textcircled{g}_m},
\end{array}$$ as asserted.\end{proof}

We are ready to prove:

\begin{thm} Let $a\in \mathcal{A}$. Then $a\in \mathcal{A}_r^{\tiny\textcircled{g}_m}$ if and only if\end{thm}
\begin{enumerate}
\item [(1)] $a\in \mathcal{A}_r^d$;
\vspace{-.5mm}
\item [(2)] There exists $x\in \mathcal{A}$ such that $$(aa_r^d)^*aa_r^dx=(aa_r^d)^*a^m.$$
\end{enumerate}
In this case, $a_r^{\tiny\textcircled{g}_m}=(a_r^d)^{m+1}x$.
\begin{proof} $\Longrightarrow $ Obviously, $a\in \mathcal{A}_r^{d}.$ In light of Theorem 2.1,
$a^m\in \mathcal{A}_r^{\tiny\textcircled{g}}$. According to~\cite[Theorem 2.6]{CM}, there exists $x\in \mathcal{A}$ such that
$((a^m)(a^m)_r^d)^*((a^m)(a^m)_r^d)x=((a^m)(a^m)_r^d)^*a^m.$ Hence, $(aa_r^d)^*aa_r^dx=(aa_r^d)^*a^m,$ as desired.

$\Longleftarrow $ By hypothesis, $$(aa_r^d)^*aa_r^dx=(aa_r^d)^*a^m.$$ for some $x\in \mathcal{A}$.
Then $$((a^m)(a^m)_r^d)^*((a^m)(a^m)_r^d)x=((a^m)(a^m)_r^d)^*a^m.$$ By virtue of ~\cite[Theorem 3.3]{CM},
$a^m\in \mathcal{A}_r^{\tiny\textcircled{g}}$. Moreover, we have
$$\begin{array}{rll}
a_r^{\tiny\textcircled{g}_m}&=&a^{m-1}(a^m)_r^{\tiny\textcircled{g}}\\
&=&a^{m-1}(a_r^d)^{2m}aa_r^dx\\
&=&[a^{m-1}(a_r^d)^{m-1}](a_r^d)^{m+1}aa_r^dx\\
&=&(a_r^d)^{m+1}x,
\end{array}$$ as asserted.\end{proof}

\begin{cor} Let $a\in \mathcal{A}_r^{\tiny\textcircled{g}_m}$. Then $a\in \mathcal{A}_r^{\tiny\textcircled{g}_{m+1}}$.
In this case, $$a_r^{\tiny\textcircled{g}_{m+1}}=[a_r^{\tiny\textcircled{g}_m}]^2a.$$\end{cor}
\begin{proof} In view of Theorem 2.4, $a\in \mathcal{A}_r^d$ and there exists $x\in \mathcal{A}$ such that $$(aa_r^d)^*aa_r^dx=(aa_r^d)^*a^m.$$
Then $(aa_r^d)^*aa_r^d(xa)=(aa_r^d)^*a^{m+1}.$ By virtue of Theorem 2.4, $a_r^{\tiny\textcircled{g}_m}=(a_r^d)^{m+1}x$ for some $x\in R$.
Then
$$\begin{array}{rll}
a_r^daa_r^{\tiny\textcircled{g}_m}&=&a_r^da(a_r^d)^{m+1}x=a_r^d[a(a_r^d)^2](a_r^d)^{m-1}x\\
&=&(a_r^d)^2(a_r^d)^{m-1}x=(a_r^d)^{m+1}x=a_r^{\tiny\textcircled{g}_m}.
\end{array}$$

By using Theorem 2.4 again, $a\in \mathcal{A}_r^{\tiny\textcircled{g}_{m+1}}$ and
$$\begin{array}{rll}
a_r^{\tiny\textcircled{g}_{m+1}}&=&(a_r^d)^{m+2}(xa)\\
&=&a_r^d[(a_r^d)^{m+1}x]a\\
&=&a_r^da_r^{\tiny\textcircled{g}_m}a\\
&=&a_r^da[a_r^{\tiny\textcircled{g}_m}]^2a\\
&=&[a_r^{\tiny\textcircled{g}_m}]^2a.
\end{array}$$ This completes the proof.\end{proof}

Recall that $a\in \mathcal{A}$ has Drazin inverse if there exists $x\in \mathcal{A}$ such that $ax^2=x, ax=xa, a^n=a^{n+1}x$ for some $n\in {\Bbb N}$. Such $x$ is unique, if it exists, and we denote it by $a^D$. We now generalize ~\cite[Proposition 3.11]{Z4} as follows.

\begin{cor} Let $a\in \mathcal{A}$. Then $a\in \mathcal{A}^{\tiny\textcircled{W}_m}$ if and only if\end{cor}
\begin{enumerate}
\item [(1)] $a\in \mathcal{A}^D$;
\vspace{.5mm}
\item [(2)] There exists $x\in \mathcal{A}$ such that $$(a^D)^*a^Dx=(a^D)^*a^m.$$
\end{enumerate}
In this case, $a^{\tiny\textcircled{W}_m}=(a^D)^{m+2}x$.
\begin{proof} This is obvious by Theorem 2.4.\end{proof}

\begin{cor} Let $a\in \mathcal{A}$. Then $a\in \mathcal{A}_r^{\tiny\textcircled{g}_m}$ if and only if\end{cor}
\begin{enumerate}
\item [(1)] $a\in \mathcal{A}_r^d$;
\vspace{-.5mm}
\item [(2)] There exists an idempotent $p\in \mathcal{A}$ such that $$a_r^d\mathcal{A}=p\mathcal{A}~\mbox{and} ~(aa_r^d)^*a^mp=(aa_r^d)^*a^{m}.$$
\end{enumerate}
In this case, $a_r^{\tiny\textcircled{g}_m}=a_r^dp.$
\begin{proof} $\Longrightarrow $ By using~\cite[Theorem 3.1]{CM}, $a^m\in \mathcal{A}_r^{d}$ and there exists an idempotent $p\in \mathcal{A}$ such that
$$a^m(a^m)_r^d\mathcal{A}=p\mathcal{A}~\mbox{and} ~(a^m)^*a^mp=p^*(a^m)^*a^m.$$ Then
$a\in mathcal{A}_r^d$ and $a_r^d=a^{m-1}(a^m)_ra^d$. Then $aa_r^d=a[a^{m-1}(a^m)_ra^d]=a^m(a^m)_r^d$. We directly check that
$$\begin{array}{rll}
paa_r^d&=&aa_r^dpaa_r^d=aa_r^d,\\
aa_r^dp&=&aa_r^daa_r^dp=p.
\end{array}$$ Moreover, we have
$$\begin{array}{rl}
&(aa_r^d)^*a^mp\\
=&((a_r^d)^m)^*[(a^m)^*a^mp]=((a_r^d)^m)^*[(a^m)^*a^mp]^*=((a_r^d)^m)^*[p^*(a^m)^*a^m]\\
=&[a^mp(a_r^d)^m]^*a^m=[a^m(paa_r^d)(a_r^d)^m]^*a^m=[a^m(aa_r^d)(a_r^d)^m]^*a^m\\
=&(aa_r^d)^*a^m.
\end{array}$$ This implies that
$(aa_r^d)^*a^mp=(aa_r^d)^*a^m$, as required.

$\Longleftarrow $ By hypothesis, there exists an idempotent $p\in \mathcal{A}$ such that $$a_r^d\mathcal{A}=p\mathcal{A}~\mbox{and}
~(aa_r^d)^*a^mp=(aa_r^d)^*a^{m}.$$
Write $p=a_r^dz$ with $z\in \mathcal{A}$. Then $$\begin{array}{rll}
(aa_r^d)^*aa_r^d(a^{m-1}z)&=&(aa_r^d)^*a^mp\\
&=&(aa_r^d)^*a^m.
\end{array}$$ In light of Theorem 2.4, $a\in \mathcal{A}_r^{\tiny\textcircled{g}_m}$ and
$$a_r^{\tiny\textcircled{g}_m}=(a_r^d)^{m+1}(a^{m-1}z)=(a_r^d)^2z=a_r^dp,$$ as asserted.\end{proof}

Recall that $a\in \mathcal{A}$ is regular if there exists $x\in \mathcal{A}$ such that $a=axa$. Such an $x$ is denoted by $a^{-}$. We now derive

\begin{thm} Let $a\in \mathcal{A}$ be regular. Then $a\in \mathcal{A}_r^{\tiny\textcircled{g}_{m+1}}$ if and only if $a\in \mathcal{A}_r^d$
and $a^2a^{-}\in \mathcal{A}_r^{\tiny\textcircled{g}_m}$. In this case, $$a_r^{\tiny\textcircled{g}_{m+1}}=[(a^2a^{-})_r^{\tiny\textcircled{g}_m}]^2a.$$\end{thm}
\begin{proof} $\Longrightarrow $ Set $x=aa_r^{\tiny\textcircled{g}_{m+1}}a^{-}$. Then we check that
$$\begin{array}{rll}
  &(a^2a^{-})x^2=(a^2a^{-})[aa_r^{\tiny\textcircled{g}_{m+1}}a^{-}]^2\\
  =&a^2(a_r^{\tiny\textcircled{g}_{m+1}})^2a^{-}]=aa_r^{\tiny\textcircled{g}_{m+1}}a^{-}=x,\\
  &\big((a^2a^{-})_r^d\big)^*(a^2a^{-})^{m+1}x\\
  =&\big((a^2a^{-})_r^d\big)^*(a^2a^{-})^{m+1}aa_r^{\tiny\textcircled{g}_{m+1}}a^{-}\\
  =&\big(a_r^daa^{-}\big)^*a^{m+2}a_r^{\tiny\textcircled{g}_{m+1}}a^{-}=(aa^{-})^*[\big(a_r^d\big)^*a^{m+2}a_r^{\tiny\textcircled{g}_{m+1}}]a^{-}\\
  =&(aa^{-})^*[\big(a_r^d\big)^*a^{m+1}]a^{-}=\big(((a^2a^{-}))_r^d\big)^*(a^2a^{-})^m,\\
  \end{array}$$
  We check that
  $$\begin{array}{rl}
  &(a^2a^{-})^n-(a^2a^{-})x(a^2a^{-})^n\\
  =&(a^2a^{-})^n-(a^2a^{-})[aa_r^{\tiny\textcircled{g}_{m+1}}a^{-}](a^2a^{-})^n\\
  =&a^2a^{-}(a^2a^{-})^{n-1}-a^2a_r^{\tiny\textcircled{g}_{m+1}}a^{-}a^2a^{-}(a^2a^{-})^{n-1}\\
  =&a[aa^{-}a^2a^{-}(a^2a^{-})^{n-2}-aa_r^{\tiny\textcircled{g}_{m+1}}aa^{-}a^2a^{-}(a^2a^{-})^{n-2}]\\
  =&a[(a^2a^{-})^{n-2}a^2-aa_r^{\tiny\textcircled{g}_{m+1}}(a^2a^{-})^{n-2}a^2]a^{-}\\
 =&a[a^n-aa_r^{\tiny\textcircled{g}_{m+1}}a^{n}]a^{-}.\\
 \end{array}$$ Hence,
$$||(a^2a^{-})^n-(a^2a^{-})x(a^2a^{-})^n||^{\frac{1}{n}}\leq ||a||^{\frac{1}{n}}||a^n-aa_r^{\tiny\textcircled{g}_{m+1}}a^{n}||^{\frac{1}{n}}||a^{-}||^{\frac{1}{n}}.$$
Since $\lim\limits_{n\to \infty}||a^n-aa_r^{\tiny\textcircled{g}_{m+1}}a^{n}||^{\frac{1}{n}}=0,$
we have
$$\lim\limits_{n\to \infty}||(a^2a^{-})^n-(a^2a^{-})x(a^2a^{-})^n||^{\frac{1}{n}}=0.$$
Accordingly, $(a^2a^{-})^{\tiny\textcircled{g}_{m}}=x$.

$\Longleftarrow$ Step 1. We verify that
$$\begin{array}{rll}
a^2a^{-}(a^2a^{-})_r^d&=&a^2a^{-}a_r^daa^{-}=a[aa^{-}a](a_r^d)^2aa^{-}\\
&=&a^2(a_r^d)^2aa^{-}=aa_r^daa^{-}.
\end{array}$$
Step 2. $(a^2a^{-})_r^d=a_r^daa^{-}$.
We verify that
$$\begin{array}{rll}
(a^2a^{-})(a_r^daa^{-})^2&=&(aa_r^daa^{-})(a_r^daa^{-})\\
&=&a(a_r^d)^2aa^{-}=a_r^daa^{-},\\
(a^2a^{-})(a_r^daa^{-})(a^2a^{-})&=&(aa_r^daa^{-})(a^2a^{-})=aa_r^da^2a^{-}=a^2a_r^daa^{-}\\
&=&(a^2a^{-})(aa_r^daa^{-})=(a^2a^{-})^2(a_r^daa^{-}).
\end{array}$$ Since $(a-aa_r^da)a^{-}a=a-aa_r^da\in \mathcal{A}^{qnil}$, by using Cline's formula (see~\cite[Theorem 2.2]{L2}), we see that
$a(a-aa_r^da)a^{-}\in \mathcal{A}^{qnil}$. Therefore $a^2a^{-}-a^2a^{-}(a_r^daa^{-})a^2a^{-}=a^2a^{-}-a^2a_r^daa^{-}\in \mathcal{A}^{qnil}$.
Therefore $(a^2a^{-})_r^d=a_r^daa^{-}$.

Step 3. By virtue of Theorem 2.4, there exists some $x\in R$ such that
$$\big(a^2a^{-}(a^2a^{-})_r^d\big)^*[a^2a^{-}(a^2a^{-})_r^d]x=\big(a^2a^{-}(a^2a^{-})_r^d\big)^*[a^2a^{-}]^m.$$
Hence, we have
$$\big(aa_r^daa^{-}\big)^*[aa_r^daa^{-}]x=\big(aa_r^daa^{-}\big)^*[a^2a^{-}]^m.$$
$$\big(aa_r^da\big)^*[aa_r^da][a^{-}x]=\big(aa_r^da\big)^*[a^2a^{-}]^m.$$
$$\big(a^2a_r^d\big)^*[aa_r^da][a^{-}x]=\big(a^2a_r^d\big)^*[a^2a^{-}]^m.$$
$$\big(a^2(a_r^d)^2\big)^*[aa_r^da][a^{-}x]=\big(a^2(a_r^d)^2\big)^*[a^2a^{-}]^m.$$
$$\big(aa_r^d\big)^*[aa_r^da][a^{-}xa]=\big(aa_r^d\big)^*a^{m+1}.$$
By virtue of Theorem again, $a\in \mathcal{A}_r^{\tiny\textcircled{g}_{m+1}}$.\end{proof}

\begin{cor} Let $a\in \mathcal{A}$ be regular. Then $a\in \mathcal{A}^{\tiny\textcircled{W}_{m+1}}$ if and only if
$a^2a^{-1}\in \mathcal{A}^{\tiny\textcircled{W}_m}$. In this case, $$a^{\tiny\textcircled{W}_{m+1}}=[(a^2a^{-})^{\tiny\textcircled{W}_m}]^2a.$$\end{cor}
\begin{proof} This is obvious by Theorem 2.8.\end{proof}

\section{equivalent characterizations}

In this section we investigate equivalent characterizations of $m$-generalized right group inverse. We now establish the relation between $m$-generalized right group inverse and $m$-generalized group decomposition. Our starting point is the following.

\begin{thm} Let $a\in \mathcal{A}$. Then $a\in \mathcal{A}_r^{\tiny\textcircled{g}_m}$ if and only if\end{thm}
\begin{enumerate}
\item [(1)] $a\in \mathcal{A}_r^d$;
\vspace{-.5mm}
\item [(2)] There exists $x\in \mathcal{A}$ such that
$$x=ax^2, [(a^m)^*a^{m+1}x]^*=(a^m)^*a^{m+1}x, \lim\limits_{n\to \infty}||a^n-axa^{n}||^{\frac{1}{n}}=0.$$
\end{enumerate} In this case, $a_r^{\tiny\textcircled{g}_m}=aa_r^dx$.
\begin{proof} $\Longrightarrow $ Let $a_r^{\tiny\textcircled{g}_m}=x$. Then $a\in \mathcal{A}_r^d$ and there exists $x\in \mathcal{A}$ such that  $$x=ax^2, (aa_r^d)^*a^{m+1}x=(aa_r^d)^*a^m, \lim\limits_{n\to \infty}||a^n-axa^{n}||^{\frac{1}{n}}=0.$$ By virtue of Theorem 2.1,
$(a^m)_r^{\tiny\textcircled{g}}=(a_r^{\tiny\textcircled{g}_m})^m=x^m$. In view of~\cite[Theorem 2.6]{CM}, we have
$\big((a^m)^*a^{2m}x^m\big)^*=(a^m)^*a^{2m}x^m.$
Clearly, $$(a^m)^*a^{2m}x^m=(a^m)^*a^m(a^mx^m)=(a^m)^*a^{m+1}x;$$ hence,
$[(a^m)^*a^{m+1}x]^*=(a^m)^*a^{m+1}x$.

$\Longleftarrow $ By hypothesis, there exists $x\in \mathcal{A}$ such that
$$x=ax^2, [(a^m)^*a^{m+1}x]^*=(a^m)^*a^{m+1}x, \lim\limits_{n\to \infty}||a^n-axa^{n}||^{\frac{1}{n}}=0.$$

Let $y=x^m$. Then $$\begin{array}{rll}
a^my^2&=&(a^mx^m)x^m=ax^{m+1}=x^m=y,\\
(a^m)^*a^{2m}y&=&(a^m)^*a^{2m}x^m=(a^m)^*a^{m}(a^mx^m)=(a^m)^*a^{m+1}x,\\
\big((a^m)^*a^{2m}y\big)^*&=&\big((a^m)^*a^{m+1}x\big)^*=(a^m)^*a^{m+1}x=(a^m)^*a^{2m}y,\\
\end{array}$$
Moreover, we have
$$(a^m)^n-(a^m)y(a^m)^{n}=(a^m)^n-a^mx^m(a^m)^{n}=a^{mn}-axa^{mn}.$$
Hence, $$\lim\limits_{n\to \infty}||(a^m)^n-(a^m)y(a^m)^{n}||^{\frac{1}{n}}=0.$$
Thus $a^m\in \mathcal{A}_r^{\tiny\textcircled{g}}$ and $(a^m)_r^{\tiny\textcircled{g}}=y$. In light of Theorem 2.1, $a\in \mathcal{A}_r^{\tiny\textcircled{g}_m}$ and
$$a_r^{\tiny\textcircled{g}_m}=a^{m-1}(a^m)_r^{\tiny\textcircled{g}}=a^{m-1}y=a^{m-1}x^m=[a^{m-1}x^{m-1}]x=ax^2=x,$$ as asserted.\end{proof}

\begin{cor} Let $a\in \mathcal{A}$. Then $a\in \mathcal{A}^{\tiny\textcircled{g}_m}$ if and only if\end{cor}
\begin{enumerate}
\item [(1)] $a\in \mathcal{A}^d$;
\vspace{-.5mm}
\item [(2)] There exists $x\in \mathcal{A}$ such that
$$x=ax^2, [(a^m)^*a^{m+1}x]^*=(a^m)^*a^{m+1}x, \lim\limits_{n\to \infty}||a^n-axa^{n}||^{\frac{1}{n}}=0.$$
\end{enumerate} In this case, $a_r^{\tiny\textcircled{g}_m}=aa_r^dx$.
\begin{proof} Evidently, $a\in \mathcal{A}^{\tiny\textcircled{g}_m}$ if and only if $a\in \mathcal{A}^d\bigcap \mathcal{A}_r^{\tiny\textcircled{g}_m}$.
Therefore we obtain the result by Theorem 3.1.\end{proof}

\begin{thm} Let $a\in \mathcal{A}$. Then the following are equivalent:\end{thm}
\begin{enumerate}
\item [(1)] $a\in \mathcal{A}_r^{\tiny\textcircled{g}_m}$.
\vspace{-.5mm}
\item [(2)] $a$ has $m$-generalized group decomposition, i.e., there exist $x,y\in \mathcal{A}$ such that $$a=x+y, x^*a^{m-1}y=yx=0, x\in \mathcal{A}_r^{\#}, y\in \mathcal{A}^{qnil}.$$
\end{enumerate}
In this case, $a_r^{\tiny\textcircled{g}_m}=x_r^{\#}$.
\begin{proof} $(1)\Rightarrow (2)$ By hypotheses, we have $z\in \mathcal{A}$ such that $$z=az^2, ((a^m)^*a^{m+1}z)^*=(a^m)^*a^{m+1}z, \lim\limits_{n\to
\infty}||a^n-aza^n||^{\frac{1}{n}}=0.$$
For any $n\in {\Bbb N}$, we have $az=a(az^2)=a^2z^2=a^2(az^2)z=a^3z^3=\cdots =a^nz^n=\cdots =a^{n+1}z^{n+1}$.
Thus, we prove that $$\begin{array}{rll}
||az-azaz||&=&||(a-aza)z||\\
&=&||(a^n-aza^{n})z^{n-1}||.
\end{array}$$ Hence, $$\begin{array}{rll}
||az-azaz||^{\frac{1}{n}}&\leq &||(a^n-aza^{n})||^{\frac{1}{n}}||z||^{1-\frac{1}{n}}.
\end{array}$$
This implies that $$\lim\limits_{n\to \infty}||az-azaz||^{\frac{1}{n}}=0,$$ whence, $az=azaz$.

Set $x=a^2z$ and $y=a-a^2z.$ Then $a=x+y$.
We check that $$\begin{array}{rcl}
(a^2-aza^2)z&=&(a^2-aza^2)az^2\\
&=&(a^2-aza^2)a^2z^3\\
&\vdots&\\
&=&(a^2-aza^2)a^{n-2}z^{n-1}\\
&=&(a^n-aza^{n})z^n.
\end{array}$$
Therefore $$||(a^2-aza^2)z||^{\frac{1}{n}}\leq ||a^n-aza^{n}|^{\frac{1}{n}}|||z||.$$
Since $$\lim\limits_{n\to \infty}||a^n-aza^{n}||^{\frac{1}{n}}=0,$$ we deduce that $$\lim\limits_{n\to \infty}||(a^2-aza^2)z||^{\frac{1}{n}}=0.$$
This implies that $(a^2-aza^2)z=0$.

We claim that $x$ has right group inverse. We easily check that
$$\begin{array}{rll}
xz^2&=&(a^2z)z^2=az^2=z,\\
xzx&=&a^2z^2a^2z=aza^2z=a^2z=x,\\
x^2z&=&a^2za^2z^2=a^2zaz=a^2z=x.
\end{array}$$ Then $x\in \mathcal{A}_r^{\#}.$

We check that
$$\begin{array}{rll}
||(a-aza)^{n+1}||^{\frac{1}{n+1}}&=&||(a-aza)^n(a-aza)||^{\frac{1}{n+1}}\\
&=&||(a-aza)^{n-1}(a-aza)a||^{\frac{1}{n+1}}\\
&=&||(a-aza)^{n-1}a^2||^{\frac{1}{n+1}}\\
&\vdots&\\
&=&||(a-aza)a^n||^{\frac{1}{n+1}}\\
&\leq &\big[||a^{n+1}-aza^{n+1}||^{\frac{1}{n}}\big]^{\frac{n}{n+1}}.
\end{array}$$ Accordingly, $$\lim\limits_{n\to \infty}||(a-aza)^{n+1}||^{\frac{1}{n+1}}=0.$$ Hence, $a-aza\in \mathcal{A}^{qnil}$. By
Cline's formula (see~\cite[Theorem 2.2]{L2}), $y=a-a^2z\in \mathcal{A}^{qnil}$.

Furthermore, we verify that $$\begin{array}{rll}
x^*a^{m-1}y&=&(a^2z)^*a^{m-1}(a-a^2z)=[a(az)]^*a^{m-1}(a-a^2z)\\
&=&[a(a^mz^m)]^*a^{m-1}(a-a^2z)=(z^{m-1})^*[(a^{m+1}z)^*a^m](1-az)\\
&=&(z^{m-1})^*[(a^m)^*a^{m+1}z]^*(1-az)=(z^{m-1})^*[(a^m)^*a^{m+1}z](1-az)\\
&=&(z^{m-1})^*(a^m)^*a^{m+1}(z-zaz)=(z^{m-1})^*(a^m)^*a^{m}a(z-zaz)=0,\\
yx&=&(a-a^2z)(a^2z)=a^3z-a^2za^2z=a^3z-a(a^2z)=0,
\end{array}$$ as desired.

$(2)\Rightarrow (1)$ By hypothesis, there exist $x,y\in \mathcal{A}$ such that $$a=x+y, x^*a^{m-1}y=yx=0, x\in \mathcal{A}_r^{\#}, y\in \mathcal{A}^{qnil}.$$
Hence, $a^m=a^{m-1}x+a^{m-1}y$. We verify that
$$\begin{array}{rll}
a^{m-1}x&=&a^{m-2}(x+y)x=a^{m-2}x^2=\cdots =x^m,\\
ya^{m-1}&=&y(x+y)a^{m-2}=y^2a^{m-2}=\cdots =y^m,\\
\big(a^{m-1}x\big)^*\big(a^{m-1}y\big)&=&(x^{m-1})^*[x^*a^{m-1}]y=0,\\
\big(a^{m-1}y\big)\big(a^{m-1}x\big)&=&\big(a^{m-1}y\big)x^m=a^{m-1}(yx)x^{m-1}=0.
\end{array}$$ Since $y\in \mathcal{A}^{qnil}$, we see that $y^m\in \mathcal{A}^{qnil}$. By
Cline's formula, $a^{m-1}y\in \mathcal{A}^{qnil}$. It is easy to verify that

$$\begin{array}{rll}
x^m[(x_r^{\#})^m]^2&=&[x^m(x_r^{\#})^m](x_r^{\#})^m=(x_r^{\#})^m=x,\\
(x^m)^2(x_r^{\#})^m&=&x^m[(x^m)(x_r^{\#})^m]=x^m[xx_r^{\#}]=x^{m-1}[x^2x_r^{\#}]=x^m,\\
x^m(x_r^{\#})^mx^m&=&xx_r^{\#}x^m=[xx_r^{\#}x]x^{m-1}=x^m.
\end{array}$$
Hence, $(x^m)_r^{\#}=(x_r^{\#})^m$. Accordingly, we have $a^{m-1}x\in \mathcal{A}_r^{\#}$ and $[a^{m-1}x]_r^{\#}=(x^m)_r^{\#}=(x_r^{\#})^m.$
By virtue of ~\cite[Theorem 2.6]{CM}, $a^m\in \mathcal{A}_r^{\tiny\textcircled{g}}$ and
$(a^m)_r^{\tiny\textcircled{g}}=[a^{m-1}x]_r^{\#}=(x_r^{\#})^m.$
According to Theorem 2.1, $a\in \mathcal{A}$ has $m$-generalized right group inverse and we have
$$\begin{array}{rll}
a_r^{\tiny\textcircled{g}_m}&=&a^{m-1}(a^m)_r^{\tiny\textcircled{g}}=a^{m-1}(x_r^{\#})^m=a^{m-2}(x+y)(x_r^{\#})^m\\
&=&a^{m-2}x(x_r^{\#})^m=x^{m-1}(x_r^{\#})^m=x_r^{\#},
\end{array}$$ as asserted.\end{proof}

\begin{cor} Let $a\in \mathcal{A}_r^{\tiny\textcircled{g}_m}$. Then the following hold:\end{cor}
\begin{enumerate}
\item [(1)] $a_r^{\tiny\textcircled{g}_m}=a_r^{\tiny\textcircled{g}_m}aa_r^{\tiny\textcircled{g}_m}$.
\vspace{-.5mm}
\item [(2)] $aa_r^{\tiny\textcircled{g}_m}=a^n(a_r^{\tiny\textcircled{g}_m})^n$ for any $n\in {\Bbb N}$.
\end{enumerate}
\begin{proof} These are obvious by the proof of Theorem 3.3.\end{proof}

As is well known, $a\in R$ has generalized Drazin inverse if and only if it has quasi-polar property, i.e., there exists an idempotent $p\in R$ such that $$a+p\in \mathcal{A}^{-1}, pa=pa\in \mathcal{A}^{qnil}$$ (see~\cite[Theorem 6.4.8]{CM1}). We come now to characterize $m$-generalized right group inverse by the polar-like property.

\begin{thm} Let $a\in \mathcal{A}$. Then the following are equivalent:\end{thm}
\begin{enumerate}
\item [(1)] $a\in \mathcal{A}_r^{\tiny\textcircled{g}_m}$.
\item [(2)] There exists an idempotent $p\in \mathcal{A}$ such that
$$\begin{array}{c}
(1-p)a(1-p)\in [(1-p)\mathcal{A}(1-p)]_r^{-1},\\
\big((a^m)^*a^{m}p\big)^*=(a^m)^*a^{m}p, ap\in \mathcal{A}^{qnil}, (1-p)\mathcal{A}=a(1-p)\mathcal{A}.
\end{array}$$
\end{enumerate}
\begin{proof} $(1)\Rightarrow (2)$ Since $a\in R_r^{\tiny\textcircled{g}_m}$, by virtue of Theorem 3.3, there exist $z,y\in \mathcal{A}$ such that $$a=z+y, z^*a^{m-1}y=yz=0, z\in \mathcal{A}_r^{\#}, y\in \mathcal{A}^{qnil}.$$ Set $x=z_r^{\#}$. The we check that $$\begin{array}{rll}
ax&=&(z+y)z_r^{\#}=zz_r^{\#},\\
ax^2&=&(ax)x=z(z_r^{\#})^2=z_r^{\#}=x.
\end{array}$$ Since $yz=0$, we see that $a^kz=a^{k-1}(y+z)z=a^{k-1}z^2=\cdots =z^{k+1}$.
As $z^*a^{m-1}y=0$, we derive that $$\begin{array}{rll}
(a^m)^*z^m&=&a^*(a^{m-1})^*z^m=[z^*a^{m-1}(y+z)]^*z^{m-1}\\
&=&[z^*a^{m-1}z]^*z^{m-1}=[z^*(a^{m-1}z)]^*z^{m-1}\\
&=&[z^*(z^{m-1}z)]^*z^{m-1}=(z^m)^*z^m.
\end{array}$$ Hence, we have
$$\begin{array}{rll}
(a^m)^*a^{m+1}x&=&(a^m)^*a^{m}(ax)=(a^m)^*a^{m}(zz_r^{\#})=(a^m)^*(a^{m}z)z_r^{\#}\\
&=&(a^m)^*(z^{m+1}z_r^{\#})=(a^m)^*z^m=(z^m)^*z^m.
\end{array}$$ Therefore $$[(a^m)^*a^{m+1}x]^*=[(z^m)^*z^m]^*=(z^m)^*z^m=(a^m)^*a^{m+1}x.$$
Let $p=1-zz_r^{\#}$. Then $p=1-ax=p^2\in \mathcal{A}$. Since $a-azz_r^{\#}=a(1-zz_r^{\#})=(z+y)(1-zz_r^{\#})=y\in \mathcal{A}^{qnil}$, by Cline's formula (see~\cite[Lemma 6.4.10]{CM1}), $pa=(1-zz_r^{\#})a=a-zz_r^{\#}a\in \mathcal{A}^{qnil}$ .
Therefore we have $$[(a^m)^*a^{m}p]^*=((a^m)^*a^{m}-[(a^m)^*a^{m+1}x]^*=(a^m)^*a^m-(a^m)^*a^{m+1}x=(a^m)^*a^{m}p.$$
Since $pa(1-p)=(1-zz_r^{\#})(z+y)zz_r^{\#}=0$, we get $pa=pap$. Then we verify that
$$[(1-p)a(1-p)][(1-p)z_r^{\#}(1-p)]=(z^2z_r^{\#})z_r^{\#}(zz_r^{\#}))=zz_r^{\#})=1-p.$$ Hence,
$(1-p)a(1-p)\in [(1-p)\mathcal{A}(1-p)]_r^{-1}$. Clearly, $ap=(z+y)(1-zz_r^{\#})=y\in \mathcal{A}^{qnil}$.
We verify that $$1-p=[(1-p)a(1-p)][(1-p)z_r^{\#}(1-p)]=a(1-p)z_r^{\#}(1-p)\in a(1-p)\mathcal{A}$$ and
$$a(1-p)=(1-p)a(1-p)\in (1-p)\mathcal{A}.$$ Therefore $(1-p)\mathcal{A}=a(1-p)\mathcal{A},$ as required.

$(2)\Rightarrow (1)$ By hypothesis, there exists an idempotent $p\in \mathcal{A}$ such that $$\begin{array}{c}
(1-p)a(1-p)\in [(1-p)\mathcal{A}(1-p)]_r^{-1},\\
\big((a^m)^*a^{m}p\big)^*=(a^m)^*a^{m}p, ap\in \mathcal{A}^{qnil}, (1-p)\mathcal{A}=a(1-p)\mathcal{A}.
\end{array}$$

Set $x=a(1-p)$ and $y=ap$. Then $$\begin{array}{rll}
a&=&x+y, yx=apa(1-p)=0,\\
y&=&ap\in \mathcal{A}^{qnil}.
\end{array}$$

Moreover, we verify that $$\begin{array}{rll}
x^*a^{m-1}y&=&[x^m(x^{\#})^{m-1}]^*a^{m-1}y=((x^{\#})^{m-1})^*(x^m)^*a^{m-1}(ap)\\
&=&[(x^{\#})^{m-1}]^*(1-p)^*[(a^m)^*a^{m}p]=[(x^{\#})^{m-1}]^*(1-p)^*[(a^m)^*a^{m}p]^*\\
&=&[(x^{\#})^{m-1}]^*(1-p)^*p^*[(a^m)^*a^{m}]^*=0.
\end{array}$$

Since $(1-p)\mathcal{A}=a(1-p)\mathcal{A}$, we have $pa(1-p)=0$. Clearly, $x=a(1-p)=(1-p)a(1-p)$. Let $z=[(1-p)a(1-p)]_r^{-1}$. Then
$$\begin{array}{rll}
x^2z&=&a(1-p)a(1-p)z=a(1-p)=x,\\
xzx&=&(1-p)a(1-p)zx=(1-p)x=x,\\
xz^2&=&[(1-p)a(1-p)z]z=(1-p)z=z.
\end{array}$$ Hence $x\in \mathcal{A}_r^{\#}$.
Therefore $a\in \mathcal{A}_r^{\tiny\textcircled{g}_m}$ by Theorem 3.3.\end{proof}

\begin{cor} Let $a\in \mathcal{A}_r^{\tiny\textcircled{g}_m}$. Then there exists an idempotent $p\in \mathcal{A}$ such that
$$a+p\in \mathcal{A}_r^{-1}, [(a^m)^*a^{m}p]^*=(a^m)^*a^{m}p, pa=pap\in \mathcal{A}^{qnil}.$$\end{cor}
\begin{proof} By virtue of Theorem 3.5, there exists an idempotent $p\in \mathcal{A}$ such that
$$\begin{array}{c}
(1-p)a(1-p)\in [(1-p)\mathcal{A}(1-p)]_r^{-1},\\
\big((a^m)^*a^{m}p\big)^*=(a^m)^*a^{m}p, ap\in \mathcal{A}^{qnil}, (1-p)\mathcal{A}=a(1-p)\mathcal{A}.
\end{array}$$ Construct $z$ and $y$ as in Theorem 3.5. Then $p=1-zz_r^{\#}=1-a_r^{\tiny\textcircled{g}_m}=p^2\in \mathcal{A}$. We verify that $$[z+1-zz_r^{\#}][z_r^{\#}+1-zz_r^{\#}]=zz_r^{\#}+z(1-zz_r^{\#})+1-zz_r^{\#}=1+z(1-zz_r^{\#})\in \mathcal{A}^{-1}.$$
Since $y(z_r^{\#}+1-zz_r^{\#})(1-z(1-zz_r^{\#}))=y\in \mathcal{A}^{qnil}$, it follows by Cline's formula that $(z_r^{\#}+1-zz_r^{\#})(1-z(1-zz_r^{\#}))y\in \mathcal{A}^{qnil}$.
Hence $1+(z_r^{\#}+1-zz_r^{\#})(1-z(1-zz_r^{\#}))y\in \mathcal{A}^{-1}$. This implies that $$\begin{array}{rll}
a+p&=&z+y+1-zz_r^{\#}\\
&=&(z+1-zz_r^{\#})[1+(z_r^{\#}+1-zz_r^{\#})(1-z(1-zz_r^{\#}))y]\\
&\in&\mathcal{A}_r^{-1},
\end{array}$$ as asserted.\end{proof}

\begin{thm} Let $a\in \mathcal{A}$. Then the following are equivalent:\end{thm}
\begin{enumerate}
\item [(1)] $a\in \mathcal{A}_{r}^{\tiny\textcircled{g}_m}$.
\item [(2)] There exists $b\in \mathcal{A}$ such that $$bab=b, a^2b^2=ab, [(a^m)^*a^{m+1}b]^*=(a^m)^*a^{m+1}b, ab\mathcal{A}=a^2b\mathcal{A}, a-a^2b\in \mathcal{A}^{qnil}.$$
\end{enumerate}
\begin{proof} $(1)\Rightarrow (2)$ By hypothesis, there exist $x,y\in \mathcal{A}$ such that $$a=x+y, x^*a^{m-1}y=yx=0, x\in \mathcal{A}_{r}^{\#},
y\in \mathcal{A}^{qnil}.$$ Set $b=x_{r}^{\#}$. Then
$$\begin{array}{rll}
bab&=&x_{r}^{\#}(x+y)x_{r}^{\#}=x_{r}^{\#}xx_{r}^{\#}=x_{r}^{\#}=b,\\
a^2b^2&=&(x+y)^2(x_{r}^{\#})^2=x^2(x_{r}^{\#})^2=xx_{r}^{\#}\\
&=&(xx_{r}^{\#})^2=(x+y)x_{r}^{\#}(x+y)x_{r}^{\#}\\
&=&(ab)^2.
\end{array}$$
Moreover, we have $$\begin{array}{rll}
ab&=&(x+y)x_r^{\#}=xx_{r}^{\#};\\
a^2b&=&(x+y)xx_r^{\#}=x^2x_r^{\#}=x.
\end{array}$$ Hence, $ab\in a^2b\mathcal{A}$, and so $ab\mathcal{A}=a^2b\mathcal{A}$. Since $a(1-xx^{\tiny\textcircled{\#}})=(x+y)(1-xx^{\tiny\textcircled{\#}})=y\in \mathcal{A}^{qnil}$, by using Cline's formula, $a-a^2b=a(1-xx^{\tiny\textcircled{\#}})\in \mathcal{A}^{qnil}.$
By hypothesis, we have
$x^*a^{m-1}y=0$ and $yx=0$, we have $$x^*[x^{m-1}+x^{m-2}y+\cdots +xy^{m-2}+y^{m-1}]y=0.$$ Hence,
$$x^*[x^{m-1}y+x^{m-2}y^2+\cdots +xy^{m-1}+y^{m}]=0.$$
Further, we verify that
$$\begin{array}{rll}
x^*(x+y)^{m+1}&=&x^*[x^{m}+x^{m-1}y+\cdots +xy^{m-1}+y^m](x+y)\\
&=&[x^*x^{m}+x^*(x^{m-1}y+\cdots +xy^{m-1}+y^m)](x+y)\\
&=&x^*x^{m+1}+x^*x^{m}y.
\end{array}$$
Therefore we have
$$\begin{array}{rl}
&(a^m)^*a^{m+1}b\\
=&((x+y)^m)^*(x+y)^{m+1}x_r^{\#}\\
=&[x^m+x^{m-1}y+\cdots +xy^{m-1}+y^m]^*(x+y)^{m+1}x_r^{\#}\\
=&[x^{m-1}+x^{m-2}y+\cdots +y^{m-1}]^*x^*(x+y)^{m+1}x_r^{\#}+[y^m]^*(x+y)^{m+1}x_r^{\#}\\
=&[x^{m-1}+x^{m-2}y+\cdots +y^{m-1}]^*[x^*x^{m+1}+x^*x^{m}y]x_r^{\#}+[y^m]^*x^{m+1}x_r^{\#}\\
=&[x^{m-1}+x^{m-2}y+\cdots +y^{m-1}]^*[x^*x^{m+1}x_r^{\#}]+[y^m]^*x^{m+1}x_r^{\#}\\
=&[x^{m}+x^{m-1}y+\cdots +xy^{m-1}]^*x^{m}+[y^m]^*x^{m}\\
=&[x^{m}+x^{m-1}y+\cdots +xy^{m-1}+y^m]^*x^{m}\\
=&(x^m)^*x^m+[x^*(x^{m-1}y+\cdots +xy^{m-1}+y^m)]^*x^{m-1}\\
=&(x^m)^*x^m.
\end{array}$$ Accordingly, $[(a^m)^*a^{m+1}b]^*=[(x^m)^*x^m]^*=(x^m)^*x^m=(a^m)^*a^{m+1}b$, as required.

$(2)\Rightarrow (1)$ By hypothesis, there exists $b\in \mathcal{A}$ such that $$bab=b, a^2b=aba, [(a^m)^*a^{m+1}b]^*=(a^m)^*a^{m+1}b, ab\mathcal{A}=a^2b\mathcal{A}, a-a^2b\in \mathcal{A}^{qnil}.$$
Let $p=1-ab$. Then $p^2=p$ and $ap=a-a^2b\in \mathcal{A}^{qnil}.$
Clearly, $(a^m)^*a^{m}p=(a^m)^*a^m(1-ab)=(a^m)^*a^m-(a^m)^*a^{m+1}b$. Hence, $\big((a^m)^*a^{m}p\big)^*=(a^m)^*a^{m}p$.
Since $ab\mathcal{A}=a^2b\mathcal{A}$, we have $(1-p)\mathcal{A}=a(1-p)\mathcal{A}$.
Moreover, we see that $(1-p)a(1-p)=aba^2b, (1-p)b(1-p)=ab^2ab$ and $(aba^2b)(ab^2ab)=aba^2b^2=ababab=ab$.
Hence $$\begin{array}{c}
(1-p)a(1-p)\in [(1-p)\mathcal{A}(1-p)]_r^{-1},\\
\big((a^m)^*a^{m}p\big)^*=(a^m)^*a^{m}p, ap\in \mathcal{A}^{qnil}, (1-p)\mathcal{A}=a(1-p)\mathcal{A}.
\end{array}$$ Therefore $a\in \mathcal{A}_{r}^{\tiny\textcircled{g}_m}$ by Theorem 3.5.\end{proof}

\begin{cor} Let $a\in \mathcal{A}$. Then the following are equivalent:\end{cor}
\begin{enumerate}
\item [(1)] $a\in \mathcal{A}_{r}^{\tiny\textcircled{g}}$.
\item [(2)] There exists $b\in \mathcal{A}$ such that $$bab=b, a^2b^2=ab, [a^*a^2b]^*=a^*a^2b, ab\mathcal{A}=a^2b\mathcal{A}, a-a^2b\in \mathcal{A}^{qnil}.$$
\end{enumerate}
\begin{proof} This is proved by choosing $m=1$ in Theorem 3.7.\end{proof}

\section{representations using generalized right core inverses}

This section is to investigate algebraic properties of $m$-generalized right group inverse. We will present the representations of $m$-generalized right group inverse by using generalized right core inverses. We now derive

\begin{thm} Let $a\in \mathcal{A}_r^{\tiny\textcircled{d}}$. Then $a\in \mathcal{A}_r^{\tiny\textcircled{g}_m}$ and
$$a_r^{\tiny\textcircled{g}_m}=(a_r^d)^{m+1}aa_r^{\tiny\textcircled{d}}a^{m}=(a_r^daa_r^{\tiny\textcircled{d}})^{m+1}a^m.$$\end{thm}
\begin{proof} Set $x=aa_r^{\tiny\textcircled{d}}a^m$. Then we verify that
$$\begin{array}{rll}
(aa_r^d)^*(aa_r^d)x&=&(aa_r^d)^*(aa_r^d)aa_r^{\tiny\textcircled{d}}a^m\\
&=&(aa_r^d)^*aa_r^{\tiny\textcircled{d}}a^m\\
&=&(aa_r^d)^*(aa_r^{\tiny\textcircled{d}})^*a^m\\
&=&(aa_r^{\tiny\textcircled{d}}aa_r^d)^*a^m\\
&=&(aa_r^d)^*a^m.
\end{array}$$ According to Theorem 2.4, $a\in \mathcal{A}_r^{\tiny\textcircled{g}_m}$ and
$$a_r^{\tiny\textcircled{g}_m}=(a_r^d)^{m+1}x=(a_r^d)^{m+1}aa_r^{\tiny\textcircled{d}}a^{m}.$$

Moreover, we have $$(a_r^daa_r^{\tiny\textcircled{d}})^{m+1}a^m=(a_r^d)^{m+1}aa_r^{\tiny\textcircled{d}}a^{m},$$ as asserted.\end{proof}

\begin{cor} Let $a\in \mathcal{A}_r^{\tiny\textcircled{d}}$. Then $a_r^{\tiny\textcircled{g}_m}=x~\mbox{if and only if} ~ax^2=x,
ax=(a_r^daa^{\tiny\textcircled{d}})^{m}a^{m}.$\end{cor}
\begin{proof} $\Longrightarrow $ In view of Theorem 4.1, $a\in \mathcal{A}_r^{\tiny\textcircled{g}_m}$ and
$x:=a_r^{\tiny\textcircled{g}_m}=(a_r^daa_r^{\tiny\textcircled{d}})^{m+1}a^m$.
Therefore $ax^2=x$ and $ax=a(a_r^daa_r^{\tiny\textcircled{d}})^{m+1}a^m=(a_r^daa^{\tiny\textcircled{d}})^{m}a^{m}$, as required.

$\Longleftarrow $ By hypothesis, $ax^2=x, ax=(a_r^daa^{\tiny\textcircled{d}})^{m}a^{m}$. Then we have
$$\begin{array}{rll}
x&=&ax^2=(ax)x=[(a_r^daa^{\tiny\textcircled{d}})^{m}a^{m}]x\\
&=&[(a_r^daa^{\tiny\textcircled{d}})^{m}a^{m-1}](ax)\\
&=&[(a_r^daa^{\tiny\textcircled{d}})^{m}a^{m-1}][(a_r^daa^{\tiny\textcircled{d}})^{m}a^{m}]\\
&=&(a_r^daa^{\tiny\textcircled{d}})^{m}a^{m-2}(aa_r^da)a^{\tiny\textcircled{d}}(a_r^daa^{\tiny\textcircled{d}})^{m-1}a^{m}\\
&=&(a_r^daa^{\tiny\textcircled{d}})^{m}a^{m-2}[aa^{\tiny\textcircled{d}}](a_r^daa^{\tiny\textcircled{d}})^{m-1}a^{m}\\
&=&(a_r^daa^{\tiny\textcircled{d}})^{m}a^{m-2}(a_r^daa^{\tiny\textcircled{d}})^{m-1}a^{m}\\
&=&(a_r^daa^{\tiny\textcircled{d}})^{m}a^{m-3}(a_r^daa^{\tiny\textcircled{d}})^{m-2}a^{m}\\
&\vdots&\\
&=&(a_r^daa^{\tiny\textcircled{d}})^{m}a_r^daa^{\tiny\textcircled{d}}a^{m}\\
&=&(a_r^daa_r^{\tiny\textcircled{d}})^{m+1}a^m.
\end{array}$$ By virtue of Theorem 4.1,
$x=a_r^{\tiny\textcircled{g}_m}$, as desired.\end{proof}

To illustrate Theorem 4.1, we present the following example.

\begin{exam} The space ${\ell}^2({\Bbb N})$ is a Hilbert space consisting of all square-summable infinite sequences of complex numbers,
let $S_n$ be defined on ${\ell}^2({\Bbb N})$ by:
$S_n(x_1,x_2,x_3,\cdots )=(\underbrace{0,0,\cdots ,0,}_{n} x_1,x_2,x_3,\cdots )$. $L_n$ is defined
on ${\ell}^2({\Bbb N})$ by:
$L_n(x_1,x_2,x_3,\cdots )=(x_{n+1},x_{n+2},x_{n+3},\cdots )$.
Then $L_nS_n=I,$ while $S_nL_n\neq I.$
Then the left shift operator $L_1$ has generalized right core inverse. Let $m\in {\Bbb N}$. By virtue of Theorem 4.1, $L_1$ has a $m$-generalized right core inverse and $$(L_1)_r^{\tiny\textcircled{g}_m}=S_{m+1}L_m,$$ which is given by
$(L_1)_r^{\tiny\textcircled{g}_m}(x_1,x_2,x_3,\cdots )=(\underbrace{0,0,\cdots ,0,}_{m+1} x_{m+1},x_{m+2},x_{m+3},\cdots )$.
\end{exam}

\begin{thm} Let $a\in \mathcal{A}_r^{\tiny\textcircled{d}}$. Then
$$a_r^{\tiny\textcircled{g}_m}=(a_r^d)^{m+2}(a_r^d)_r^{\tiny\textcircled{\#}}a^{m}.$$\end{thm}
\begin{proof} By virtue of~\cite[Theorem 3.5]{CM2}, we have
$$a_r^{\tiny\textcircled{d}}=(a_r^d)^2(a_r^d)^{\tiny\textcircled{\#}}.$$ According to Theorem 4.1, we derive that
$$\begin{array}{rll}
a_r^{\tiny\textcircled{g}_m}&=&(a_r^daa_r^{\tiny\textcircled{d}})^{m+1}a^m\\
&=&[a_r^da(a_r^d)^2(a_r^d)^{\tiny\textcircled{\#}}]^{m+1}a^m\\
&=&(a_r^d)^2(a_r^d)^{\tiny\textcircled{\#}}[a_r^da(a_r^d)^2(a_r^d)^{\tiny\textcircled{\#}}]^{m}a^m\\
&=&(a_r^d)^3(a_r^d)^{\tiny\textcircled{\#}}[a_r^da(a_r^d)^2(a_r^d)^{\tiny\textcircled{\#}}]^{m-1}a^m\\
&\vdots& \\
&=&(a_r^d)^{m+2}(a_r^d)_r^{\tiny\textcircled{\#}}a^{m},
\end{array}$$ as asserted.\end{proof}

\begin{cor} Let $a\in \mathcal{A}^{\tiny\textcircled{D}}$. Then
$$a^{\tiny\textcircled{W}_m}=(a^D)^{2}(a^D)^{\tiny\textcircled{\#}}.$$\end{cor}
\begin{proof} This is immediate by Theorem 4.4.\end{proof}

Recall that an element $a\in \mathcal{A}$ has right core inverse if there exist $x\in \mathcal{A}$ such that $$ax^2=x, (ax)^*=ax, axa=a.$$ If such $x$ exists, it is unique, and denote it by $a_r^{\tiny\textcircled{\#}}$. We refer the reader more properties of right core in ~\cite{WM}.

\begin{thm} Let $a\in \mathcal{A}_r^{\tiny\textcircled{d}}$. Then
$$a_r^{\tiny\textcircled{g}_m}=[a_r^da^m(a^{m+1}a_r^{\tiny\textcircled{\#}})_r^{\tiny\textcircled{\#}}]^{m+1}a^m.$$\end{thm}
\begin{proof} We easily verifies that
$$\begin{array}{rll}
(a^{m+1}a_r^{\tiny\textcircled{d}})(a_r^{\tiny\textcircled{d}})^m&=&a^{m+1}(a_r^{\tiny\textcircled{d}})^{m+1}=aa_r^{\tiny\textcircled{d}},\\
(a^{m+1}a_r^{\tiny\textcircled{d}})[(a_r^{\tiny\textcircled{d}})^m]^2&=&aa_r^{\tiny\textcircled{d}}(a_r^{\tiny\textcircled{d}})^m=(a_r^{\tiny\textcircled{d}})^m,\\
(a^{m+1}a_r^{\tiny\textcircled{d}})(a_r^{\tiny\textcircled{d}})^m(a^{m+1}a_r^{\tiny\textcircled{d}})&=&a^{m+1}a_r^{\tiny\textcircled{d}}.
\end{array}$$
Hence, $$[a^{m+1}a_r^{\tiny\textcircled{\#}}]_r^{\tiny\textcircled{\#}}=(a_r^{\tiny\textcircled{d}})^m.$$ By virtue of Theorem 4.1, we derive that
$$\begin{array}{rll}
a_r^{\tiny\textcircled{g}_m}&=&(a_r^daa_r^{\tiny\textcircled{d}})^{m+1}a^m\\
&=&[a_r^da^m(a_r^{\tiny\textcircled{d}})^m]^{m+1}a^m\\
&=&[a_r^da^m(a^{m+1}a_r^{\tiny\textcircled{\#}})_r^{\tiny\textcircled{\#}}]^{m+1}a^m,
\end{array}$$ as asserted.\end{proof}

\begin{cor} Let $a\in \mathcal{A}^{\tiny\textcircled{d}}$. Then
$$a^{\tiny\textcircled{g}}=[(a^2a^{\tiny\textcircled{\#}})^{\tiny\textcircled{\#}}]^2a.$$\end{cor}
\begin{proof} This is obvious by Theorem 4.6.\end{proof}

We are ready to prove:

\begin{thm} Let $a\in \mathcal{A}_r^{\tiny\textcircled{d}}$ and $x\in \mathcal{A}$. Then the following are equivalent:\end{thm}
\begin{enumerate}
\item [(1)] $a_r^{\tiny\textcircled{g}_m}=x$.
 \vspace{-.5mm}
\item [(2)] $x=ax^2, a^{m+1}x=aa_r^{\tiny\textcircled{d}}a^{m}, \lim\limits_{n\to \infty}||a^n-axa^n||^{\frac{1}{n}}=0.$
\end{enumerate}
\begin{proof} $(1)\Rightarrow (2)$ By hypothesis, there exists $x\in \mathcal{A}$ such that
$$x=ax^2, (aa_r^d)^*a^{m+1}x=(aa_r^d)^*a^m, \lim\limits_{n\to \infty}||a^n-axa^n||^{\frac{1}{n}}=0.$$ In view of~\cite[Theorem 3.5]{CM2},
$$a_r^{\tiny\textcircled{d}}=(a_r^d)^2(a_r^d)_r^{\tiny\textcircled{\#}}.$$ Hence,
$$(aa_r^{\tiny\textcircled{d}})^*a^{m+1}x=(aa_r^{\tiny\textcircled{d}})^*a^m.$$ Since $(aa_r^{\tiny\textcircled{d}})^*=aa_r^{\tiny\textcircled{d}}$, we have
$$(aa_r^{\tiny\textcircled{d}})a^{m+1}x=(aa_r^{\tiny\textcircled{d}})a^m.$$ Thus
$$a_r^{\tiny\textcircled{d}}a^{m+1}x=a_r^{\tiny\textcircled{d}}[aa_r^{\tiny\textcircled{d}}a^{m+1}x]=a_r^{\tiny\textcircled{d}}
[aa_r^{\tiny\textcircled{d}}a^m]=a_r^{\tiny\textcircled{d}}a^m.$$ Hence,
$$aa_r^{\tiny\textcircled{d}}a^{m+1}x=aa_r^{\tiny\textcircled{d}}a^m.$$
Since $x=ax^2$, we deduce that $a^{m+1}x=aa_r^{\tiny\textcircled{d}}a^m,$ as required.

$(2)\Rightarrow (1)$ By hypothesis, we have $$x=ax^2, a^{m+1}x=aa_r^{\tiny\textcircled{d}}a^m.$$
Then $aa_r^{\tiny\textcircled{d}}a^{m+1}x=aa_r^{\tiny\textcircled{d}}a^m$. Hence,
$(aa_r^{\tiny\textcircled{d}})^*a^{m+1}x=(aa_r^{\tiny\textcircled{d}})^*a^m$.
Accordingly, we have $$(aa_r^d)^*(aa_r^{\tiny\textcircled{d}})^*a^{m+1}x=(aa_r^d)^*(aa_r^{\tiny\textcircled{d}})^*a^m.$$
This implies that $$(aa_r^{\tiny\textcircled{d}}aa_r^d)^*a^{m+1}x=(aa_r^{\tiny\textcircled{d}}aa_r^d)^*a^m.$$
It is easy to verify that $aa_r^{\tiny\textcircled{d}}aa_r^d=aa_r^d$; hence, $$(aa_r^d)^*a^{m+1}x=(aa_r^d)^*a^m.$$
Therefore $a_r^{\tiny\textcircled{g}_m}=x$ by Theorem 2.4.\end{proof}

\begin{cor} Let $a\in \mathcal{A}_r^{\tiny\textcircled{d}}$. Then the following are equivalent:\end{cor}
\begin{enumerate}
\item [(1)] $a\in \mathcal{A}_r^{\tiny\textcircled{g}}$.
 \vspace{-.5mm}
\item [(2)] There exists $x\in \mathcal{A}$ such that $$x=ax^2, a^2x=aa_r^{\tiny\textcircled{d}}a, \lim\limits_{n\to \infty}||a^n-axa^{n}||^{\frac{1}{n}}=0.$$
\end{enumerate}
\begin{proof} We obtain the result by choosing $m=1$ in Theorem 4.8.\end{proof}

As an immediate consequence of Theorem 4.8, we derive

\begin{cor} Let $A,X\in {\Bbb C}^{n\times n}$. Then the following are equivalent:\end{cor}
\begin{enumerate}
\item [(1)] $A^{\tiny\textcircled{W}_m}=X$.
 \vspace{-.5mm}
\item [(2)] $X=AX^2, A^{m+1}X=AA^{\tiny\textcircled{\dag}}A^{m}, A^n=AXA^n$ for some $n\in {\Bbb N}$.
\end{enumerate}

Let $a,b,c\in R$. An element $a$ has $(b,c)$-inverse if there exists $x\in R$ such that $$xab=b, cax=c ~\mbox{and}~ x\in bRx\bigcap xRc.$$ If such $x$ exists, it is unique and is denoted by $a^{(b,c)}$ (see~\cite{D1}).

\begin{thm} Let $a\in R_r^{\tiny\textcircled{d}}$. Then
$$a_r^{\tiny\textcircled{g}_m}=a^{\big((a_r^{d})^{m+1}a^m, a_r^daa_r^{\tiny\textcircled{d}}a^m\big)}.$$\end{thm}
\begin{proof} Obviously, $a\in R_r^{d}$. Let $x=a_r^{\tiny\textcircled{g}_m}$. Then we verify
that
$$\begin{array}{rll}
x&=&(a_r^d)^{m+1}aa_r^{\tiny\textcircled{d}}a^{m}\\
&=&[(a_r^d)^{m+1}a^{m}]a[(a_r^d)^{m+1}a]a_r^{\tiny\textcircled{d}}a^{m}]\\
&\in &(a_r^{d})^{m+1}a^mRx,\\
x&=&(a_r^daa_r^{\tiny\textcircled{d}})^{m+1}a^m\\
&=&(a_r^daa_r^{\tiny\textcircled{d}})^{m}a_r^d[aa_r^{\tiny\textcircled{d}}a^m]\\
&=&(a_r^daa_r^{\tiny\textcircled{d}})^{m}(a_r^daa_r^{\tiny\textcircled{d}})aa_r^d[aa_r^{\tiny\textcircled{d}}a^m]\\
&=&(a_r^daa_r^{\tiny\textcircled{d}})^{m}(a_r^daa_r^{\tiny\textcircled{d}})a^m(a_r^d)^m[aa_r^{\tiny\textcircled{d}}a^m]\\
&=&[(a_r^daa_r^{\tiny\textcircled{d}})^{m+1}a^m](a_r^d)^{m-1}[a_r^daa_r^{\tiny\textcircled{d}}a^m]\\
&\in& xR[a_r^daa_r^{\tiny\textcircled{d}}a^m],\\
xa(a_r^{d})^{m+1}a^m&=&(a_r^d)^{m+1}aa_r^{\tiny\textcircled{d}}a^{m}a(a_r^{d})^{m+1}a^m\\
&=&(a_r^d)^{m+1}aa_r^{d}a^m\\
&=&(a_r^{d})^{m+1}a^m,\\
a_r^daa_r^{\tiny\textcircled{d}}a^max&=&a_r^daa_r^{\tiny\textcircled{d}}a^ma[(a_r^d)^{m+1}aa_r^{\tiny\textcircled{d}}a^{m}]\\
&=&a_r^da^{m+1}(a_r^d)^{m+1}aa_r^{\tiny\textcircled{d}}a^{m}\\
&=&a_r^daa_r^daa_r^{\tiny\textcircled{d}}a^{m}\\
&=&a_r^daa_r^{\tiny\textcircled{d}}a^{m}.
\end{array}$$ Therefore $$a_r^{\tiny\textcircled{g}_m}=a^{\big((a_r^{d})^{m+1}a^m, a_r^daa_r^{\tiny\textcircled{d}}a^m\big)},$$ as desired.
\end{proof}

Let $a\in R$ and $T,S\subseteq R$. We say that $a$ has $\{2\}$-inverse $x$ provided that $xax=x, im(a)=T, ker(a)=S$. We denote $x$ by
$a^{(2)}_{T,S}$. We present the representation of $m$-generalized right group inverse by the $\{2\}$-inverse in a Banach algebra.

\begin{thm} Let $a\in R_r^{\tiny\textcircled{g}_m}$. Then $$a_r^{\tiny\textcircled{g}_m}=a^{(2)}_{im\big((a_r^d)^{m+1}a^{m}\big), ker(a_r^{\tiny\textcircled{d}}a^m)}.$$\end{thm}
\begin{proof} Let $x=a_r^{\tiny\textcircled{g}_m}$. Clearly, we have $x=xax$.

Step 1. $im(x)=im\big((a_r^d)^{m+1}a^{m}\big)$. In view of Theorem 4.1, we have

$$\begin{array}{rll}
x&=&(a_r^d)^{m+1}aa_r^{\tiny\textcircled{d}}a^{m}\\
&=&[(a_r^d)^{m+1}a^m](a_r^{\tiny\textcircled{d}})^ma^{m}\\
&\in&im\big((a_r^d)^{m+1}a^{m}\big),
\end{array}$$
$$\begin{array}{rll}
(a_r^d)^{m+1}a^{m}&=&(a_r^d)^{m+1}a^{m}=(a_r^d)^{m+1}aa_r^da^{m}=(a_r^d)^{m+1}a^{m+1}a_r^d\\
&=&[(a_r^d)^{m+1}aa_r^{\tiny\textcircled{d}}a^{m+1}]a_r^d\\
&=&[(a_r^d)^{m+1}aa_r^{\tiny\textcircled{d}}a^{m}][aa_r^d].
\end{array}$$
Hence $(a_r^d)^{m+1}a^{m}=x[aa_r^d]\in im(x)$. Therefore $im(x)=im\big(aa_r^{d}a^*\big)$.

Step 2. $ker(x)=ker(a_r^{\tiny\textcircled{d}}a^m)$. If $(a_r^{\tiny\textcircled{d}}a^m)r=0$ for some $r\in R$, then $$\begin{array}{rll}
xr&=&[(a_r^d)^{m+1}aa_r^{\tiny\textcircled{d}}a^{m}]r\\
&=&(a_r^d)^{m+1}a[a_r^{\tiny\textcircled{d}}a^{m}]r\\
&=&0.
\end{array}$$ Thus $r\in ker(x)$.

If $x(r)=0$, then $$\begin{array}{rll}
[a_r^{\tiny\textcircled{d}}a^{m}]r&=&(a^2)^*r=aa_r^d[a_r^{\tiny\textcircled{d}}a^{m}]r\\
&=&a^{m+1}[(a_r^d)^{m+1}a_r^{\tiny\textcircled{d}}a^{m}]r\\
&=&a^{m}[(a_r^d)^{m+1}aa_r^{\tiny\textcircled{d}}a^{m}]r\\
&=&a^mx(r)=0.
\end{array}$$ Hence, $r\in ker(a_r^{\tiny\textcircled{d}}a^m)$, as desired.

Therefore we complete the proof.\end{proof}

\begin{cor} Let $A,X\in {\Bbb C}^{n\times n}$. Then
$$A^{\tiny\textcircled{W}_m}=A^{\big((A^{D})^{m+1}A^m, A^DAA^{\tiny\textcircled{D}}A^m\big)}=A^{(2)}_{im\big((A^D)^{m+1}A^{m}\big), ker(A^{\tiny\textcircled{D}}A^m)}.$$
\end{cor}
\begin{proof} This is immediate from Theorem 4.11 and Theorem 4.12.\end{proof}

We next use the preceding presentations of $m$-generalized right group inverse in solving the matrix equations.
We consider the following equation in $\mathcal{A}$:

$$[aa_r^d]^*a^{m+1}x=[aa_r^d]^*a^{m}b,~~~~~~~~~~~~~~~~~~~~~(4.1)$$

where $a,b\in \mathcal{A}$ and $m\in {\Bbb N}$.

\begin{thm} Let $a\in \mathcal{A}_r^{\tiny\textcircled{d}}$. Then the general solution of Eq. (4.1) is
$$x=a^{\tiny\textcircled{g}_m}b+[1-a^{\tiny\textcircled{g}_m}a]y,$$
where $y\in \mathcal{A}$ is arbitrary.
\end{thm}
\begin{proof} Step 1. Let $a\in \mathcal{A}_r^{\tiny\textcircled{g}_m}$. We claim that Eq. (4.1) has solution
$$x=a_r^{\tiny\textcircled{g}_m}b+[1-a_r^{\tiny\textcircled{g}_m}a]y,$$
where $y\in \mathcal{A}$ is arbitrary.
Let $x=a_r^{\tiny\textcircled{g}_m}b+[1-a_r^{\tiny\textcircled{g}_m}a]y,$
where $y\in \mathcal{A}$. Since $[aa_r^d]^*a^{m+1}a^{\tiny\textcircled{g}_m}=[aa_r^d]^*a^{m},$ we verify that

$$\begin{array}{rll}
[aa_r^d]^*a^{m+1}x&=&[aa_r^d]^*a^{m+1}a^{\tiny\textcircled{g}_m}b+[aa_r^d]^*a^{m+1}[1-a^{\tiny\textcircled{g}_m}a]y\\
&=&[aa_r^d]^*a^{m+1}a^{\tiny\textcircled{g}_m}b+\big([aa_r^d]^*a^{m+1}-[aa_r^d]^*a^{m+1}a^{\tiny\textcircled{g}_m}a\big)y\\
&=&[aa_r^d]^*a^{m}b+\big([aa_r^d]^*a^{m+1}-[aa_r^d]^*a^{m}a\big)y\\
&=&[aa_r^d]^*a^{m}b,
\end{array}$$ as claimed.

Step 2. Let $x$ be the solution of Eq. (4.1). Then $[aa_r^d]^*a^{m+1}x=[aa_r^d]^*a^{m}b.$
In view of Theorem 4.1, $a^{\tiny\textcircled{g}_m}=[a_r^daa_r^{\tiny\textcircled{d}}]^{m+1}a^{m}$.
Then
$$\begin{array}{rll}
a^{\tiny\textcircled{g}_m}ax&=&[a_r^daa_r^{\tiny\textcircled{d}}]^{m+1}a^{m}(ax)=[a_r^daa_r^{\tiny\textcircled{d}}]^{m+1}a^{m+1}x\\
&=&[a_r^daa_r^{\tiny\textcircled{d}}]^{m}[a_r^da^2(a_r^{\tiny\textcircled{d}})^2]a^{m+1}x\\
&=&[a_r^daa_r^{\tiny\textcircled{d}}]^{m}[a_r^da^2a_r^{\tiny\textcircled{d}}][a_r^{\tiny\textcircled{d}}a^{m+1}x]\\
&=&[a_r^daa_r^{\tiny\textcircled{d}}]^{m}[a_r^da^2a_r^{\tiny\textcircled{d}}]a_r^{\tiny\textcircled{d}}[aa_r^{\tiny\textcircled{d}}]a^{m+1}x\\
&=&[a_r^daa_r^{\tiny\textcircled{d}}]^{m}[a_r^da^2a_r^{\tiny\textcircled{d}}]a_r^{\tiny\textcircled{d}}[aa_r^{\tiny\textcircled{d}}]^*a^{m+1}x\\
&=&[a_r^daa_r^{\tiny\textcircled{d}}]^{m}[a_r^da^2a_r^{\tiny\textcircled{d}}]a_r^{\tiny\textcircled{d}}\big([aa_r^d]^*a^{m+1}x\big)\\
&=&[a_r^daa_r^{\tiny\textcircled{d}}]^{m}[a_r^da^2a_r^{\tiny\textcircled{d}}]a_r^{\tiny\textcircled{d}}\big([aa_r^d]^*a^{m}b\big)\\
&=&[a_r^daa_r^{\tiny\textcircled{d}}]^{m}a_r^da^2(a_r^{\tiny\textcircled{d}})^2[aa_r^d]^*a^{m}b\\
&=&[a_r^daa_r^{\tiny\textcircled{d}}]^{m}a_r^d(aa_r^{\tiny\textcircled{d}})^*[aa_r^d]^*a^{m}b\\
&=&[a_r^daa_r^{\tiny\textcircled{d}}]^{m}a_r^d(aa_r^daa_r^{\tiny\textcircled{d}})^*a^{m}b\\
&=&[a_r^daa_r^{\tiny\textcircled{d}}]^{m}a_r^d(aa_r^{\tiny\textcircled{d}})^*a^{m}b\\
&=&[a_r^daa_r^{\tiny\textcircled{d}}]^{m}a_r^d(aa_r^{\tiny\textcircled{d}})a^{m}b\\
&=&[a_r^daa_r^{\tiny\textcircled{d}}]^{m+1}a^{m}b=a^{\tiny\textcircled{g}_m}b.
\end{array}$$
Therefore $x=a^{\tiny\textcircled{g}_m}b+[1-a^{\tiny\textcircled{g}_m}a]x,$ as required.\end{proof}

\begin{cor} Let $a\in \mathcal{A}_r^{\tiny\textcircled{d}}$. If $x$ is the solution of Eq. (4.1) in $im\big(a_r^{\tiny\textcircled{g}_m}\big)$, then $x=a_r^{\tiny\textcircled{g}_m}b.$
\end{cor}
\begin{proof} By virtue of Theorem 4.14, $a^{\tiny\textcircled{g}_m}b$ is a solution of Eq. (4.1) in $im\big(a_r^{\tiny\textcircled{g}_m}\big)$.
Let $x_1,x_2\in \mathcal{A}$ be the solutions of Eq. (4.1) and satisfy $im(x_i)\subseteq im\big(a_r^{\tiny\textcircled{g}_m}\big)$. Write $x_1=a_r^{\tiny\textcircled{g}_m}y_1$ and $x_2=a_r^{\tiny\textcircled{g}_m}y_2$. Then
$x_1-x_2=a_r^{\tiny\textcircled{g}_m}(y_1-y_2)$; hence, $im(x_1-x_2)\subseteq im\big(a_r^{\tiny\textcircled{g}_m}\big)$. By hypothesis, we have
$$[aa_r^d]^*a^{m+1}x_i=[aa_r^d]^*a^{m}b$$ for $i=1,2$. Hence $[aa_r^d]^*a^{m+1}(x_1-x_2)=0$.
Since $aa_r^da=a^2a_r^d$, by induction, we have $[aa_r^d]a^{m+1}=a^{m+1}[aa_r^d]$. Hence $$[aa_r^d]^*[aa_r^d]a^{m+1}(x_1-x_2)=[aa_r^d]^*a^{m+1}[aa_r^d](x_1-x_2)=0.$$
Since the involution is proper, we have
$[aa_r^d]a^{m+1}(x_1-x_2)=0$, and then $a^{m+1}(x_1-x_2)=a^{m+1}aa_r^d(x_1-x_2)=0$. In light of Theorem 4.1,
$$a_r^{\tiny\textcircled{g}_m}a(x_1-x_2)=(a_r^daa_r^{\tiny\textcircled{d}})^{m+1}a^{m+1}(x_1-x_2)=0.$$
By hypothesis, we may write $x_1-x_2=a_r^{\tiny\textcircled{g}_m}z$ for some $z\in \mathcal{A}$. Then
$$x_1-x_2=a_r^{\tiny\textcircled{g}_m}a[a_r^{\tiny\textcircled{g}_m}z]=a_r^{\tiny\textcircled{g}_m}a(x_1-x_2)=0.$$
Accordingly, $x=a^{\tiny\textcircled{g}_m}b$ is the unique solution of Eq. (4.1) in $im\big(a_r^{\tiny\textcircled{g}_m}\big)$.\end{proof}

{\bf Conflict of interest}

No potential conflict of interest was reported by the authors.\\

{\bf Data Availability Statement}

No/Not applicable (this manuscript does not report data generation or analysis).


\vskip10mm\begin{thebibliography}{99}
\bibitem{C1} H. Chen, On $m$-generalized group inverse in Banach *-algebras, {\it Mediterr. J. Math.}, {\bf 22}(2025),  https://doi.org/10.1007/s00009-025-02818-1.

    \bibitem{CM1} H. Chen and M. Sheibani, {\it Theory of Clean Rings and Matrices}, World Scientific, Hackensack, NJ, 2023. https://doi.org/10.1142/12959.

\bibitem{CM} H. Chen and M. Sheibani, Generalized right core inverse in *-Banach algebras, Preprints 2024, 2024061246. https://doi.org/10.20944/preprints202406.1246.v1.

\bibitem{CM2} H. Chen and M. Sheibani, Generalized right group inverse in Banach *-algebras, Preprint, 2025.

\bibitem{C3} H. Chen; D. Liu and M. Sheibani, Group invertibility of the sum in rings and its applications, {\it Georgian Math. J.}, {\bf 31}(2024), 923--932.

\bibitem{CY} Y. Chen; L. Yuan and X. Zhang, New characterizations and representations of the weak group inverse, {\it Filomat}, {\bf 38}(2024), 6703--6713.

\bibitem{D1} M.P. Drazin, A class of outer generalized inverses, {\it Linear Algebra Appl.}, {\bf 436}(2012),
1909--1923.

\bibitem{D} D.E. Ferreyra; V. Orquera and N. Thome, A weak group inverse for rectangular matrices, {\it Rev. R. Acad. Cienc. Exactas Fis.Nat., Ser. A Mat.}, {\bf 113}(2019), 3727--3740.

\bibitem{G} J. Gao; K. Zu and Q. Wang, A $m$-weak group inverse for rectangular matrices,
arXiv: 2312.10704v1 [math.R.A] 17 Dec 2023.

\bibitem{GC} Y. Gao and J. Chen, Pseudo core inverses in rings with involution, {\it Comm. Algebra}, {\bf 46}(2018), 38--50.

\bibitem{J} W. Jiang and K. Zuo, Further characterizations of the $m$-weak group inverse of a complex matrix,
{\it AIMS Math.}, {\bf 7}(2020), 17369--17392.

\bibitem{L} W. Li; J. Chen and Y. Zhou, Characterizations and representations of weak core inverses and $m$-weak group inverses,
{\it Turk. J. Math.}, {\bf 47}(2023), 1453--1468.

\bibitem{L2} Y. Liao; J. Chen and J. Cui, Cline's formula for the generalized Drazin inverse, {\it Bull. Malays. Math. Sci. Soc.}, {\bf 37}(2014), 37--42.

\bibitem{M} N. Mihajlovic, Group inverse and core inverse in Banach and $C^*$-algebras,
{\it Comm. Algebra}, {\bf 48}(2020), 1803--1818.

\bibitem{M0} D. Mosi\'c, Core-EP inverses in Banach algebras,
{\it Linear Multilinear Algebra}, {\bf 69}(2021), 2976--2989.

\bibitem{M1} D. Mosi\'c; P.S. Staninirovic, Representations for the weak group inverse, {\it Appl. Math. Comput.}, {\bf 397}(2021), Article ID 125957, 19 p.

\bibitem{M2}  D. Mosi\'c; P.S. Staninirovic and L.A. Kazakovtsev, Application of $m$-weak group inverse in solving optimization problems,
{\it Rev. R. Acad. Cienc. Exactas Fis. Nat., Ser. A Mat.}, {\bf 118}(2024), No. 1, Paper No. 13, 21 p.

\bibitem{M3} D. Mosi\'c; P.S. Stanimirovi\'c and L.A. Kazakovtsev, Minimization problem solvable by weighted $m$-weak group
inverse, {\it J. Applied Math. Computing}, {\bf 70}(2024), 6259--6281.

\bibitem{M4} D. Mosi\'c and D. Zhang, Weighted weak group inverse for Hilbert space operators, {\it Front. Math. China}, {\bf 15}(2020), 709--726.

\bibitem{M5} D. Mosi\'c and D. Zhang, New representations and properties of the $m$-weak group inverse,
{\it Result. Math.}, {\bf 78}(2023). https://doi.org/10.1007/s00025-023-01878-7.

\bibitem{R} Y. Ren and L. Jiang, Left and right-Drazin inverses in rings and operator algebras,  {\it J. Algebra Appl.}, {\bf 23}(2024).
  https://doi.org/10.1142/S0219498824500646.

\bibitem{W} H. Wang; J. Chen, Weak group inverse, {\it Open Math.}, {\bf 16}(2018), 1218--1232.

\bibitem{WM} L. Wang; D. Mosi\'c and Y. Gao, Right core inverse and the related generalized inverses,
{\it Commun. Algebra}, {\bf 47}(2019), 4749--4762.

 \bibitem{Y} K. Yan, One-sided Drazin inverses in Banach algebras and perturbations of B-Fredholm spectra, preprint, arXiv:2312.107042402.13574v1 [math.RA] 21 Feb 2024.

\bibitem{YW} H. Yan; H. Wang; K. Zuo and Y. Chen, Further characterizations of the weak group inverse of matrices and the weak group matrix, {\it AIMS Math.}, {\bf 6}(2021), 9322--9341.

\bibitem{Z2} M. Zhou; J. Chen and Y. Zhou, Weak group inverses in proper *-rings,
{\it J. Algebra Appl.}, {\bf 19}(2020), Article ID 2050238, 14 p.

\bibitem{Z3} M. Zhou; J. Chen; Y. Zhou and N. Thome, Weak group inverses and partial isometries in proper *-rings,
{\it Linear Multilinear Algebra},  {\bf 70}(2021), 1--16.

\bibitem{Z4} Y. Zhou; J. Chen and M. Zhou, $m$-Weak group inverses in a ring with involution,
{\it Rev. R. Acad. Cienc. Exactas F\'is. Nat., Ser. A Mat.}, {\bf 115}(2021), Paper No. 2, 12 p.

\end{thebibliography}
\end{document}